\documentclass{amsart}
\usepackage{graphicx, color}
\usepackage{amscd}
\usepackage{amsmath}
\usepackage{amsfonts}
\usepackage{amssymb}
\usepackage{mathrsfs}

\newtheorem{theorem}{Theorem}
\newtheorem{lemma}[theorem]{Lemma}
\newtheorem{proposition}[theorem]{Proposition}

\newtheorem{step}{Step}
\newtheorem{stepbis}{Step}

\numberwithin{equation}{section}
\newcommand{\RR}{\mathbb R}
\newcommand{\NN}{\mathbb N}

\renewcommand{\le}{\leqslant}
\renewcommand{\leq}{\leqslant}

\renewcommand{\geq}{\geqslant}

\begin{document}

\title[Nontrivial solutions of superlinear nonlocal problems]{Nontrivial solutions of superlinear\\ nonlocal problems}

\thanks{The authors were supported by the SRA Program P1-0292-0101 {\it Topology, Geometry and Nonlinear Analysis}. The first and the third author were supported by the INdAM-GNAMPA Project 2015 {\it Modelli ed equazioni non-locali di tipo frazionario}. The third author was supported by the MIUR National Research Project {\it Variational and Topological Methods in the Study of Nonlinear
Phenomena} and by the ERC grant $\epsilon$ ({\it Elliptic Pde's and
Symmetry of Interfaces and Layers for Odd Nonlinearities}).}

\author[G. Molica Bisci]{Giovanni Molica Bisci}
\address{Dipartimento PAU,
          Universit\`a `Mediterranea' di Reggio Calabria,
          Via Melissari 24, 89124 Reggio Calabria, Italy}
\email{\tt gmolica@unirc.it}

\author[D. Repov\v{s}]{Du\v{s}an Repov\v{s}}
\address{Faculty of Education, and Faculty of Mathematics and
         Physics,
         University of Ljubljana,
         POB 2964, 1001 Ljubljana, Slovenia}
\email{\tt dusan.repovs@guest.arnes.si}

\author[R. Servadei]{Raffaella Servadei}
\address{Dipartimento di Scienze di Base e Fondamenti,
          Universit\`a degli Studi di Urbino `Carlo Bo',
          Piazza della Repubblica 13, 61029 Urbino (Pesaro e Urbino), Italy}
\email{\tt raffaella.servadei@uniurb.it}



\begin{abstract}
We study the question of the existence of infinitely many weak solutions for nonlocal equations of fractional Laplacian type with homogeneous Dirichlet boundary data, in presence of a superlinear term.
Starting from the well-known Ambrosetti-Rabinowitz condition, we  consider different growth assumptions on the nonlinearity, all of superlinear type. We obtain three different existence results in this setting by using the Fountain Theorem, which extend some classical results for semilinear Laplacian equations to the nonlocal fractional setting.
\end{abstract}

\maketitle

\tableofcontents

\section{Introduction and main results}\label{sec:introduzione}
Recently, nonlocal fractional problems have been appearing in the literature in many different contexts, both in the pure mathematical research and in concrete real-world applications.
Indeed, fractional and nonlocal operators appear in many diverse fields such as optimization, finance, phase transitions, stratified materials, anomalous diffusion, crystal dislocation, soft thin films, semipermeable membranes, flame propagation, conservation laws, ultra-relativistic limits of quantum mechanics, quasi-geostrophic flows, multiple scattering, minimal
surfaces, materials science and water waves.

In this paper we are interested in the existence of infinitely many solutions of the following problem
\begin{equation}\label{problemaKlambda}
 \left\{
\begin{array}{ll}
-\mathcal L_K u-\lambda u=f(x, u) & \mbox{in }\,\, \Omega\\
u=0  & \mbox{in }\,\, \RR^n\setminus \Omega\,.
\end{array}\right.
\end{equation}
Here $\Omega$ is an open bounded subset of $\RR^n$ with continuous boundary $\partial\Omega$, $n>2s$, $s\in(0,1)$, the term $f$ satisfies different superlinear conditions, and $\mathcal L_K$ is the integrodifferential operator defined as
follows
\begin{equation}\label{lk}
\mathcal L_Ku(x):=
\int_{\RR^n}\Big(u(x+y)+u(x-y)-2u(x)\Big)K(y)\,dy\,,
\,\,\,\,\, x\in \RR^n\,,
\end{equation}
where the kernel $K:\RR^n\setminus\{0\}\to(0,+\infty)$ is such that
\begin{equation}\label{kernel}
m K\in L^1(\RR^n),\,\, \mbox{where}\,\,  m(x)=\min \{|x|^2,
1\}\,;
\end{equation}
\begin{equation}\label{kernelfrac}
\mbox{there exists}\,\, \theta>0\,\,
\mbox{such that}\,\, K(x)\geq \theta |x|^{-(n+2s)}\,\,
\mbox{for any}\,\, x\in \RR^n \setminus\{0\}\,.\\
\end{equation}
A model for $K$ is given by the singular kernel $K(x)=|x|^{-(n+2s)}$ which gives rise to the fractional Laplace operator $-(-\Delta)^s$, defined as
$$-(-\Delta)^s u(x):=
\int_{\RR^n}\frac{u(x+y)+u(x-y)-2u(x)}{|y|^{n+2s}}\,dy\,,
\,\,\,\,\, x\in \RR^n\,.$$

Under superlinear and subcritical conditions on $f$, the authors proved in \cite{svmountain, svlinking} the existence of a nontrivial solution of \eqref{problemaKlambda} for any $\lambda \in \RR$, as an application of the Mountain Pass Theorem and the Linking Theorem (see \cite{ar, rabinowitz}). Motivated by these existence results, in this paper we shall study the existence of infinitely many solutions of \eqref{problemaKlambda}, using the Fountain Theorem due to Bartsch (see \cite{bartsch}).

\subsection{Variational formulation of the problem}\label{subsec:weakproblem}
In order to study problem~\eqref{problemaKlambda}, we shall consider its weak formulation, given by \begin{equation}\label{weakproblemlambda}
\left\{\begin{array}{l}
\displaystyle\int_{\RR^n\times\RR^n} (u(x)-u(y))(\varphi(x)-\varphi(y))K(x-y) dx\,dy-\lambda \int_\Omega u(x)\varphi(x)\,dx\\
\qquad \qquad \qquad \qquad \qquad \qquad \qquad \qquad \qquad \displaystyle=\int_\Omega f(x,u(x))\varphi(x)dx \\
\qquad \qquad \qquad \qquad \qquad \qquad \qquad \qquad \qquad \qquad \qquad \qquad \qquad \qquad\forall\,\varphi \in X_0\\
u\in X_0,
\end{array}
\right.
\end{equation}
which represents the Euler-Lagrange equation of the energy functional~$\mathcal J_{K,\,\lambda}:X_0\to \RR$ defined as
\begin{equation}\label{JK}
\begin{aligned}
\mathcal J_{K,\,\lambda}(u)& :=\frac 1 2 \int_{\RR^n\times \RR^n}|u(x)-u(y)|^2 K(x-y)\,dx\,dy
-\frac \lambda 2 \int_\Omega |u(x)|^2\,dx\\
& \qquad \qquad \qquad \qquad \qquad \qquad \qquad \qquad -\int_\Omega F(x,u(x))\,dx\,,
\end{aligned}
\end{equation}
where the function $F$ is
the primitive of $f$ with respect to the second variable, that is
\begin{equation}\label{F}
{\displaystyle F(x,t)=\int_0^t f(x,\tau)d\tau}\,.
\end{equation}

Here, the space $X_0$ is defined as
$$X_0:=\big\{g\in X : g=0\,\, \mbox{a.e. in}\,\,
\RR^n\setminus \Omega\big\}\,,$$
where the functional space $X$ denotes the linear space of Lebesgue
measurable functions from $\RR^n$ to $\RR$ such that the restriction
of any function $g$ in $X$ to $\Omega$ belongs to $L^2(\Omega)$ and
$$\mbox{the map}\,\,\,
(x,y)\mapsto (g(x)-g(y))\sqrt{K(x-y)}\,\,\, \mbox{is in}\,\,\,
L^2\big((\RR^n\times\RR^n) \setminus ({\mathcal C}\Omega\times
{\mathcal C}\Omega), dxdy\big)\,,$$ with ${\mathcal C}\Omega:=\RR^n
\setminus\Omega$.

\subsection{The main results of the paper}\label{subsec:maintheorems}
Throughout this paper we shall assume different superlinear conditions on the term $f$.
First of all, we suppose that $f:\overline \Omega\times \RR\to \RR$ is a function satisfying the following standard conditions
\begin{equation}\label{fcontinua}
f\in C(\overline \Omega \times \RR)
\end{equation}
\begin{equation}\label{crescita}
\begin{aligned}
&\mbox{there exist}\,\, a_1, a_2>0\,\,\mbox{and}\,\, q\in (2, 2^*), 2^*=2n/(n-2s)\,,\,\, \mbox{such that}\\
&\qquad \qquad |f(x,t)|\le a_1+a_2|t|^{q-1}\,\, \mbox{for any}\,\, x\in \Omega, t\in \RR\,;
\end{aligned}
\end{equation}
and
\begin{equation}\label{mu0}
\begin{aligned}
&\mbox{there exist}\,\, \mu>2\,\, \mbox{and}\,\,r>0\,\, \mbox{such that}\,\, \mbox{for any}\,\, x\in \Omega, t\in \RR,\, |t|\geq r\\
& \qquad \qquad \qquad \qquad 0<\mu F(x,t)\le tf(x,t)\,,
\end{aligned}
\end{equation}
where $F$ is the function from \eqref{F}.

When looking for infinitely many solutions, it is natural to require some symmetry of the nonlinearity. Here, we assume the following condition:
\begin{equation}\label{fsimmetrica}
\begin{aligned}
f(x,-t)=-f(x,t)\,\,\,\mbox{for any}\,\, x\in \Omega,\, t\in \RR\,.
\end{aligned}
\end{equation}

As a model for $f$ we can take the function $f(x,t)=a(x)|t|^{q-2}t$, with $a\in C(\overline\Omega)$ and $q\in (2, 2^*)$\,.

The first result of this paper is in the following theorem:
\begin{theorem}\label{thmain1lambda}
Let $s\in (0,1)$, $n>2s$, and let $\Omega$ be an open bounded subset of
$\RR^n$ with continuous boundary. Let
$K:\RR^n\setminus\{0\}\to(0,+\infty)$ be a function
satisfying \eqref{kernel} and \eqref{kernelfrac} and let $f:\overline \Omega \times \RR \to \RR$ be a
function satisfying \eqref{fcontinua}--\eqref{fsimmetrica}.

Then for any $\lambda \in \RR$ the problem~\eqref{problemaKlambda} has infinitely many solutions $u_j\in X_0$, $j\in \NN$, whose energy $\mathcal J_{K,\,\lambda}(u_j)\to +\infty$ as $j\to +\infty$.
\end{theorem}

Assumption~\eqref{mu0} is the well-known {\em Ambrosetti-Rabinowitz condition}, originally introduced in \cite{ar}. This condition is often considered when dealing with superlinear elliptic boundary value problems (see, for instance, \cite{struwe, willem} and the references therein). Its importance is due to the fact that it assures the boundedness of the Palais--Smale sequences for the energy functional associated with the problem under consideration.

The Ambrosetti-Rabinowitz condition is a superlinear growth assumption on the nonlinearity~$f$.
Indeed, integrating \eqref{mu0} we get that
\begin{equation}\label{Fmu}
\begin{aligned}
&\qquad \mbox{there exist}\,\, a_3, a_4>0\,,\,\, \mbox{such that}\\
& F(x,t)\geq a_3 |t|^{\mu}-a_4\,\,\, \mbox{for any}\,\, (x,t)\in \overline\Omega\times  \RR\,,
\end{aligned}
\end{equation}
see, for instance, \cite[Lemma~4]{svlinking} for a detailed proof.

As a consequence of~\eqref{Fmu} and the fact that $\mu>2$, we have that
\begin{equation}\label{Flimite}
\displaystyle \lim_{|t|\to +\infty}\frac{F(x,t)}{|t|^2}=+\infty\,\,\, \mbox{uniformly for any}\,\, x\in \overline\Omega\,,
\end{equation}
which is another superlinear assumption on $f$ at infinity.

A simple computation proves that the function
\begin{equation}\label{fAR}
f(x,t)=t\log(1+|t|)
\end{equation}
satisfies condition~\eqref{Flimite}, but not~\eqref{Fmu} (and so, as a consequence, does not satisfy \eqref{mu0}).

Recently, many superlinear problems without the Ambrosetti-Rabinowitz condition have been considered in the literature (see, for instance, \cite{CostaM, fmz, jeanjean, miyagaki, schechter, WZ, Z} and references therein). In particular, in \cite{alvesliu, fangliu, liu, liu2, liuliu} the local analogue of problem~\eqref{problemaKlambda} (that is problem~\eqref{problemaKlambda} with $\mathcal L_K$ replaced by $-(-\Delta)$) has been studied. In this framework Jeanjean introduced in \cite{jeanjean} the following assumption on $f$:
\begin{equation}\label{mathcalF}
\begin{aligned}
& \qquad \mbox{there exists}\,\, \gamma\geq 1\,\, \mbox{such that for any}\,\,\, x\in \Omega\\
& \mathcal F(x,t') \leq \gamma \mathcal F(x,t)\,\, \mbox{for any}\,\, t, t'\in \RR\,\, \mbox{with}\,\, 0<t'\leq t\,,
\end{aligned}
\end{equation}
where
\begin{equation}\label{mathcalFdef}
\mathcal F(x,t)=\frac 1 2\, tf(x,t)-F(x,t)\,.
\end{equation}
Note that \eqref{mathcalF} is a global condition and that the function~\eqref{fAR} satisfies \eqref{mathcalF}.

In this setting our existence result becomes:
\begin{theorem}\label{thmain2lambda}
Let $s\in (0,1)$, $n>2s$, and let $\Omega$ be an open bounded subset of
$\RR^n$ with continuous boundary. Let
$K:\RR^n\setminus\{0\}\to(0,+\infty)$ be a function
satisfying \eqref{kernel} and \eqref{kernelfrac} and let $f:\overline \Omega \times \RR \to \RR$ be a
function satisfying \eqref{fcontinua}, \eqref{crescita}, \eqref{fsimmetrica}, \eqref{Flimite} and \eqref{mathcalF}.

Then for any $\lambda \in \RR$ the problem~\eqref{problemaKlambda} has infinitely many solutions $u_j\in X_0$, $j\in \NN$, whose energy $\mathcal J_{K,\,\lambda}(u_j)\to +\infty$ as $j\to +\infty$.
\end{theorem}

Another interesting condition used in the classical Laplace framework is the following one introduced by Liu in \cite{liu2}:
\begin{equation}\label{fmonotonar}
\begin{aligned}
& \qquad \qquad \qquad \mbox{there exists}\,\, \bar t>0\,\, \mbox{such that for any}\,\,\, x\in \Omega\\
&\mbox{the function}\,\,\, t\mapsto \frac{f(x,t)}{t}\,\,\, \mbox{is increasing if}\,\, t\geq \bar t\,\,\,
\mbox{and decreasing if}\,\, t\leq -\bar t.
\end{aligned}
\end{equation}

Under this assumption, our main result reads as follows:
\begin{theorem}\label{thmain3lambda}
Let $s\in (0,1)$, $n>2s$, and let $\Omega$ be an open bounded subset of
$\RR^n$ with continuous boundary. Let
$K:\RR^n\setminus\{0\}\to(0,+\infty)$ be a function
satisfying \eqref{kernel} and \eqref{kernelfrac} and let $f:\overline \Omega \times \RR \to \RR$ be a
function satisfying \eqref{fcontinua}, \eqref{crescita}, \eqref{fsimmetrica}, \eqref{Flimite} and \eqref{fmonotonar}.

Then for any $\lambda \in \RR$ the problem~\eqref{problemaKlambda} has infinitely many solutions $u_j\in X_0$, $j\in \NN$, whose energy $\mathcal J_{K,\,\lambda}(u_j)\to +\infty$ as $j\to +\infty$.
\end{theorem}

We remark that, due to the symmetry assumption~\eqref{fsimmetrica}, if $u$ is a weak solution of problem~\eqref{problemaKlambda}, then so is $-u$. Hence, our results give the existence of infinitely many pairs $\{u_j, -u_j\}_{j\in \NN}$ of weak solutions of \eqref{problemaKlambda}.

The proofs of Theorems~\ref{thmain1lambda}, \ref{thmain2lambda} and \ref{thmain3lambda} rely on the same arguments as used in \cite{molicaservadeizhang}. In certain steps of the proofs we only need some careful estimates of the term ${\displaystyle \lambda\|u\|_{L^2(\Omega)}^2}$.
More precisely, the strategy of our proofs consists in looking for infinitely many critical points for the energy functional associated with problem~\eqref{problemaKlambda}, namely here we apply the Fountain Theorem proved by Bartsch in \cite{bartsch}. For this purpose, we have to analyze the compactness properties of the functional and its geometric features. As for the compactness, when the nonlinearity satisfies the Ambrosetti-Rabinowitz assumption~\eqref{mu0}, we shall prove that the Palais--Smale condition is satisfied; when $f$ is assumed to satisfy the conditions \eqref{Flimite} and \eqref{mathcalF} or \eqref{fmonotonar}, the Cerami condition will be considered. In both cases the main difficulty is related to the proof of the boundedness of the Palais--Smale (or Cerami) sequence.

The geometry of the functional required by the Fountain Theorem consists in proving that the functional $\mathcal J_{K,\,\lambda}$ is negative in a ball in a suitable finite-dimensional subspace of $X_0$ and positive in a ball in an infinite-dimensional subspace.

Theorem~\ref{thmain1lambda} is the nonlocal analogue of \cite[Corollary~3.9]{willem}, where the limit case as $s\to 1$ (that is the Laplace case) was considered. Also, in \cite{servadeiKavian} the existence of infinitely many solutions of \eqref{problemaKlambda} was proved under assumptions on $f$ which were different from the ones considered here and only for the case when $q\in (2, 2^*-2s/(n-2s))$ in \eqref{crescita}, but in presence of a perturbation $h\in L^2(\Omega)$.

This paper is organized as follows. In Section~\ref{sec:preliminaries} we shall recall some preliminary notions and results. In Section~\ref{sec:AR} we shall discuss problem~\eqref{problemaKlambda} under the Ambrosetti-Rabinowitz condition and we shall prove Theorem~\ref{thmain1lambda}.  Section~\ref{sec:noAR} will be devoted to problem~\eqref{problemaKlambda} without the Ambrosetti-Rabinowitz condition and the proof of Theorem~\ref{thmain2lambda} and of Theorem~\ref{thmain3lambda} will be provided.

\section{Preliminaries}\label{sec:preliminaries}
This section is devoted to some preliminary results.
First of all, the functional space~$X_0$ we shall work in is endowed with the norm
\begin{equation}\label{normaX0}
X_0\ni g\mapsto \|g\|_{X_0}:=\left(\int_{\RR^n\times \RR^n}|g(x)-g(y)|^2K(x-y)\,dx\,dy\right)^{1/2}
\end{equation}
and $\left(X_0, \|\cdot\|_{X_0}\right)$ is a Hilbert space (for this see \cite[Lemma~7]{svmountain}), with the following scalar product
\begin{equation}\label{scalarproduct}
\langle u,v\rangle_{X_0}:=\int_{\RR^n\times \RR^n}
\big( u(x)-u(y)\big) \big( v(x)-v(y)\big)\,K(x-y)\,dx\,dy\,.
\end{equation}

The usual fractional
Sobolev space $H^s(\Omega)$ is endowed with the so-called \emph{Gagliardo norm} (see, for instance \cite{adams, valpal}) given by
\begin{equation}\label{gagliardonorm}
\|g\|_{H^s(\Omega)}:=\|g\|_{L^2(\Omega)}+
\Big(\int_{\Omega\times
\Omega}\frac{\,\,\,|g(x)-g(y)|^2}{|x-y|^{n+2s}}\,dx\,dy\Big)^{1/2}\,.
\end{equation}
Note that, even in the model case in which $K(x)=|x|^{-(n+2s)}$, the norms
in \eqref{normaX0} and \eqref{gagliardonorm}
are not the same: this makes the space $X_0$ not equivalent to the usual fractional Sobolev spaces and the classical fractional Sobolev space approach is not sufficient for studying our problem from a variational point of view.

We recall that by \cite[Lemma~5.1]{sv} the space $X_0$ is non-empty, since $C^2_0 (\Omega)\subseteq X_0$ and that
for a general kernel $K$ satisfying conditions~\eqref{kernel} and \eqref{kernelfrac}, the following inclusion holds
$$X_0\subseteq \big\{g\in H^s(\RR^n) : g=0 \,\,\, \mbox{a.e. in}\,\,\, \RR^n\setminus \Omega\big\}\,,$$
while in the model case $K(x)=|x|^{-(n+2s)}$\,, the following characterization is valid:
$$X_0=\big\{g\in H^s(\RR^n) : g=0 \,\,\, \mbox{a.e. in}\,\,\, \RR^n\setminus \Omega\big\}\,.$$
For further details on $X$ and $X_0$ we refer to \cite{sv, svmountain,
svlinking, servadeivaldinociBN}, where various properties of these spaces were
proved; for more details on the fractional Sobolev spaces $H^s$ we refer
to~\cite{valpal} and the references therein.

Finally, we recall that the eigenvalue problem driven by $-\mathcal L_K$, namely
\begin{equation}\label{problemaautovalori}
\left\{\begin{array}{ll}
-\mathcal L_K u=\lambda u & \mbox{in } \Omega
\\
u=0 & \mbox{in } \RR^{n}\setminus\Omega\,,
\end{array}
\right.
\end{equation}
possesses a divergent sequence of positive eigenvalues
$$\lambda_1<\lambda_2\le \dots \le \lambda_k\le \lambda_{k+1}\le \dots\,,$$
whose corresponding eigenfunctions will be denoted by $e_k$\,. By \cite[Proposition~9]{svlinking}, we know that $\{e_k\}_{k\in \NN}$ can be chosen in such a way that this sequence provides an orthonormal basis in $L^2(\Omega)$ and an orthogonal basis in $X_0$\,. Further properties of the spectrum of the operator~$-\mathcal L_K$ can be found in \cite[Proposition~2.3]{sY}, \cite[Proposition~9 and Appendix~A]{svlinking} and \cite[Proposition~4]{servadeivaldinociBNLOW}\,.

\subsection{The Fountain Theorem}\label{subsec:fountain}
In order to prove our main results, stated in Theorem~\ref{thmain1lambda}, Theorem~\ref{thmain2lambda} and Theorem~\ref{thmain3lambda}, in the sequel we shall apply the Fountain Theorem due to Bartsch (see \cite{bartsch}), which, under suitable compactness and geometric assumptions on a functional, provides the existence of an unbounded sequence of critical value for it.

The compactness condition assumed in the Fountain Theorem is the well-known {\em Palais--Smale condition} (see, for instance, \cite{struwe, willem}), which in our framework reads as follows:
\begin{center}
$\mathcal J_{K,\,\lambda}$ satisfies the {\em Palais--Smale compactness condition} at level $c\in \RR$\\
if any sequence $\{u_j\}_{j\in\NN}$ in $X_0$ such that\\
$\mathcal J_{K,\,\lambda}(u_j)\to c \,\, \mbox{and}\,\, \sup\Big\{
\big|\langle\,\mathcal J_{K,\,\lambda}'(u_j),\varphi\,\rangle \big|\,: \;
\varphi\in X_0\,,
\|\varphi\|_{X_0}=1\Big\}\to 0$
as $j\to +\infty$,\\
admits a strongly convergent subsequence in $X_0$\,.
\end{center}

In \cite{cerami1, cerami2} Cerami introduced the so-called {\em Cerami condition}, as a weak version of the Palais--Smale assumption. With our notation, it can be written as follows:
\begin{center}
$\mathcal J_{K,\,\lambda}$ satisfies the {\em Cerami compactness condition} at level $c\in \RR$\\
if any sequence $\{u_j\}_{j\in\NN}$ in $X_0$ such that\\
$\mathcal J_{K,\,\lambda}(u_j)\to c \,\, \mbox{and}\,\, (1+\|u_j\|)\sup\Big\{
\big|\langle\,\mathcal J_{K,\,\lambda}'(u_j),\varphi\,\rangle \big|\,: \;
\varphi\in X_0\,,
\|\varphi\|_{X_0}=1\Big\}\to 0$\\
as $j\to +\infty$,
admits a strongly convergent subsequence in $X_0$\,.
\end{center}

When the right-hand side $f$ of problem~\eqref{problemaKlambda} satisfies the Ambrosetti-Rabinowitz condition, we shall prove in the sequel that the corresponding energy functional~$\mathcal J_{K,\,\lambda}$ satisfies the Palais--Smale compactness assumption, while, when we remove the Ambrosetti-Rabinowitz condition~\eqref{mu0} and we replace it with assumptions \eqref{Flimite} and \eqref{mathcalF} or \eqref{fmonotonar}, we shall show that $\mathcal J_{K,\,\lambda}$ satisfies the Cerami condition.

Following the notation in \cite[Theorem~2.5]{bartsch} (see also \cite{willem}), we put in the sequel for any $k\in \NN$
$$Y_k:=\mbox{span}\{e_1, \ldots,e_k\}$$
and
$$Z_k:=\overline{\mbox{span}\{e_k, e_{k+1}, \dots \}}\,.$$
Note that, since $Y_k$ is finite-dimensional, all norms on $Y_k$ are equivalent and this will be used in the sequel.

Thanks to these notations, the geometric assumptions of the Fountain Theorem in our framework read as follows:
\begin{itemize}
\item[$(i)$] $a_{k}:=\max\Big\{\mathcal J_{K,\,\lambda}(u): u\in Y_k, \|u\|_{X_0}=r_{k}\Big\}\leq 0$;
\item[$(ii)$] $b_{k}:=\inf\Big\{\mathcal J_{K,\,\lambda}(u): u\in Z_k, \|u\|_{X_0}=\gamma_k\Big\}\to \infty$ as $k\to \infty$.
\end{itemize}

\section{Nonlinearities satisfying the Ambrosetti-Rabinowitz condition}\label{sec:AR}
This section is devoted to problem~\eqref{problemaKlambda} in presence of a nonlinear term satisfying condition~\eqref{mu0}.
In this framework we shall prove the following result about the compactness of the functional~$\mathcal J_{K,\,\lambda}$:
\begin{proposition}\label{propositionps1}
Let $\lambda\in \RR$ and $f:\overline \Omega \times \RR \to \RR$ be a function satisfying conditions
\eqref{fcontinua}--\eqref{mu0}. Then $\mathcal J_{K,\,\lambda}$ satisfies the Palais--Smale condition at any level $c\in \RR$\,.
\end{proposition}

\begin{proof}
Let $c\in \RR$ and let $\{u_j\}_{j\in\NN}$ be a
sequence in $X_0$ such that
\begin{equation}\label{Jc0}
\mathcal J_{K,\,\lambda}(u_j)\to c
\end{equation}
and
\begin{equation}\label{J'00}
\sup\Big\{ \big|\langle\,\mathcal J_{K,\,\lambda}'(u_j),\varphi\,\rangle \big|\,: \;
\varphi\in
X_0\,, \|\varphi\|_{X_0}=1\Big\}\to 0
\end{equation}
as $j\to +\infty$.

We split the proof into two steps. First, we show that the sequence $\{u_j\}_{j\in\NN}$ is bounded in $X_0$ and then that it admits a strongly convergent subsequence in $X_0$\,.
In showing the boundedness of the sequence $\{u_j\}_{j\in\NN}$ we have to treat separately the case when the parameter $\lambda\leq 0$ and $\lambda>0$.

\begin{step}\label{step1ps} The sequence $\{u_j\}_{j\in\NN}$ is bounded in $X_0$\,.
{\rm For any $j\in \NN$ it easily follows by \eqref{Jc0} and \eqref{J'00}
that there exists $\kappa>0$ such that
$$\Big|\langle \mathcal J_{K,\,\lambda}'(u_j), \frac{u_j}{\|u_j\|_{X_0}}\rangle\Big| \leq \kappa$$
and
$$|\mathcal J_{K,\,\lambda}(u_j)|\leq \kappa\,,$$
so that
\begin{equation}\label{addPS}
\mathcal J_{K,\,\lambda}(u_j)-\frac 1 \mu \langle \mathcal J_{K,\,\lambda}'(u_j), u_j\rangle\leq \kappa \left(1+ \|u_j\|_{X_0}\right)\,,
\end{equation}
where $\mu$ is the parameter given by \eqref{mu0}.

By invoking \eqref{crescita} and integrating it is easily seen that for any $x\in \overline\Omega$ and for any $t\in \RR$
\begin{equation}\label{cond22F}
|F(x,t)|\leq a_1\,|t|+\frac{a_2}{q}\, |t|^{q}\,.
\end{equation}
Hence, by \eqref{cond22F} and again by \eqref{crescita}, we have that for any $j\in \NN$
\begin{equation}\label{uj<r}
\begin{aligned} & \Big|\int_{\Omega \cap\{|u_j|\leq
\,r\}}\Big(F(x, u_j(x))-\frac 1 \mu\, f(x,
u_j(x))\, u_j(x)\,\Big)\,dx\Big|\\
& \qquad \qquad \qquad \leq \left(a_1 r+\frac{a_2}{q}\,r^q+\frac{a_1}{\mu} r+
\frac{a_2}{\mu}\, r^q\right)|\Omega|=:\tilde \kappa\,.
\end{aligned}
\end{equation}

Now, assume that $\lambda\leq 0$. Then, thanks to \eqref{mu0} and \eqref{uj<r}, we get that
\begin{equation}\label{jj'0}
\begin{aligned}
\mathcal J_{K,\,\lambda}(u_j)-\frac 1 \mu \langle \mathcal J_{K,\,\lambda}'(u_j), u_j\rangle & = \left(\frac 1 2 -\frac 1 \mu\right)\left(\|u_j\|_{X_0}^2-\lambda\|u_j\|_{L^2(\Omega)}^2\right)\\
& \qquad  -\int_\Omega \Big(F(x, u_j(x))-\frac 1 \mu f(x, u_j(x)) \,u_j(x)\Big)\,dx\\
& \geq \left(\frac 1 2 -\frac 1 \mu\right)\|u_j\|_{X_0}^2\\
& \qquad  -\int_\Omega \Big(F(x, u_j(x))-\frac 1 \mu f(x, u_j(x)) \,u_j(x)\Big)\,dx\\
& \geq \left(\frac 1 2 -\frac 1 \mu\right)\|u_j\|_{X_0}^2\\
& \qquad -\int_{\Omega \cap\{|u_j|\leq r\}}\Big(F(x, u_j(x))-\frac 1 \mu\, f(x, u_j(x))\,u_j(x)\Big)\,dx\\
& \geq \left(\frac 1 2 -\frac 1 \mu\right)\|u_j\|_{X_0}^2-\tilde
\kappa
\end{aligned}
\end{equation}
for any $j\in \NN$\,.

By \eqref{addPS}, \eqref{jj'0} and the fact that $\mu>2$ we have that
$$\left(\frac 1 2 -\frac 1 \mu\right)\|u_j\|_{X_0}^2 \leq \kappa\left(1+\|u_j\|_{X_0}\right)+\tilde \kappa$$
for any $j\in \NN$\,, that is $\{u_j\}_{j\in \NN}$ is bounded in $X_0$.

Now, let us consider the case when $\lambda>0$: the argument is the same as above, even though a more careful analysis is required. For reader's convenience, we prefer to give all the details.

First of all, let us fix $\sigma\in (2, \mu)$, where $\mu>2$ is given in assumption~\eqref{mu0}\,.
Arguing as above we get that for any $j\in \NN$
\begin{equation}\label{addaddPS}
\mathcal J_{K,\,\lambda}(u_j)-\frac 1 \sigma \langle \mathcal J_{K,\,\lambda}'(u_j), u_j\rangle\leq \kappa \left(1+ \|u_j\|_{X_0}\right)\,
\end{equation}
and
\begin{equation}\label{uj<rL}
\begin{aligned} & \Big|\int_{\Omega \cap\{|u_j|\leq
\,r\}}\Big(F(x, u_j(x))-\frac 1 \sigma\, f(x,
u_j(x))\, u_j(x)\,\Big)\,dx\Big| \leq \tilde \kappa\,,
\end{aligned}
\end{equation}
for suitable positive $\kappa$ and $\tilde \kappa$\,.
Then, using \eqref{mu0}, \eqref{Fmu} and \eqref{uj<rL}, we have that for any $j\in \NN$
\begin{equation}\label{jj'0L}
\begin{aligned}
\mathcal J_{K,\,\lambda}(u_j)-\frac 1 \sigma \langle \mathcal J_{K,\,\lambda}'(u_j), u_j\rangle & = \left(\frac 1 2 -\frac 1 \sigma\right)\Big(\|u_j\|_{X_0}^2-\lambda \|u_j\|_{L^2(\Omega)}^2\Big)\\
& \quad  -\int_\Omega \Big(F(x, u_j(x))-\frac 1 \sigma\, f(x, u_j(x)) \,u_j(x)\Big)\,dx\\
& \geq \left(\frac 1 2 -\frac 1 \sigma\right)\Big(\|u_j\|_{X_0}^2-\lambda \|u_j\|_{L^2(\Omega)}^2\Big)\\
& \quad +\left(\frac \mu \sigma-1\right)\int_{\Omega \cap\{|u_j|\geq\, r\}} F(x, u_j(x))\,dx\\
& \quad -\int_{\Omega \cap\{|u_j|\leq\, r\}}\Big(F(x, u_j(x))-\frac 1 \sigma\, f(x, u_j(x))\,u_j(x)\Big)\,dx\\
& \geq \left(\frac 1 2 -\frac 1 \sigma\right)\Big(\|u_j\|_{X_0}^2-\lambda \|u_j\|_{L^2(\Omega)}^2\Big)\\
& \quad +\left(\frac \mu \sigma-1\right)\int_{\Omega \cap\{|u_j|\geq\, r\}} F(x, u_j(x))\,dx-\tilde
\kappa\\
& \geq \left(\frac 1 2 -\frac 1 \sigma\right)\Big(\|u_j\|_{X_0}^2-\lambda \|u_j\|_{L^2(\Omega)}^2\Big)\\
& \quad +a_3\left(\frac \mu \sigma-1\right)\|u_j\|_{L^\mu(\Omega)}^\mu-a_4\left(1 - \frac \mu \sigma\right)\,|\Omega|-\tilde\kappa\,.
\end{aligned}
\end{equation}
Furthermore, for any $\varepsilon>0$ the Young inequality (with
conjugate exponents~$\mu/2>1$
and~$\mu/(\mu-2)$) yields
\begin{equation}\label{youngPS}
\|u_j\|_{L^2(\Omega)}^2\leq \frac{2\varepsilon}{\mu}\,\|u_j\|_{L^\mu(\Omega)}^\mu+\frac{\mu-2}{\mu}\,\varepsilon^{-2/(\mu-2)}\,|\Omega|\,,
\end{equation}
so that, by \eqref{jj'0L} and \eqref{youngPS}, we can deduce that for any $j\in \NN$
\begin{equation}\label{young2}
\begin{aligned}
\mathcal J_{K,\,\lambda}(u_j)-\frac 1 \sigma \langle \mathcal J_{K,\,\lambda}'(u_j), u_j\rangle &
\geq \left(\frac 1 2 -\frac 1 \sigma\right)\|u_j\|_{X_0}^2-\lambda\left(\frac 1 2 -\frac 1 \sigma\right) \frac{2\varepsilon}{\mu}\,\|u_j\|_{L^\mu(\Omega)}^\mu\\
& \qquad -\lambda\left(\frac 1 2 -\frac 1 \sigma\right) \frac{\mu-2}{\mu}\,\varepsilon^{-2/(\mu-2)}\,|\Omega|\\
& \qquad +a_3\left(\frac \mu \sigma-1\right)\|u_j\|_{L^\mu(\Omega)}^\mu-a_4\left(1 - \frac \mu \sigma\right)\,|\Omega|-\tilde\kappa\\
& = \left(\frac 1 2 -\frac 1 \sigma\right)\|u_j\|_{X_0}^2\\
& \qquad +\Big[a_3\left(\frac \mu \sigma-1\right)-\lambda\left(\frac 1 2 -\frac 1 \sigma\right)\frac{2\varepsilon}{\mu}\Big]\|u_j\|_{L^\mu(\Omega)}^\mu-C_\varepsilon\,,
\end{aligned}
\end{equation}
where $C_\varepsilon$ is a constant such that $C_\varepsilon\to +\infty$ as $\varepsilon\to 0$, due to $\mu>\sigma>2$\,.

Now, choosing $\varepsilon$ so small that
$$a_3\left(\frac \mu \sigma-1\right)-\lambda\left(\frac 1 2 -\frac 1 \sigma\right)\frac{2\varepsilon}{\mu}>0\,,$$
by \eqref{young2}, we get for any $j\in \NN$
\begin{equation}\label{jj'20}
\mathcal J_{K,\,\lambda}(u_j)-\frac 1 \sigma \langle \mathcal J_{K,\,\lambda}'(u_j), u_j\rangle \geq \left(\frac 1 2 -\frac 1 \sigma\right)\|u_j\|_{X_0}^2-C_\varepsilon\,.
\end{equation}

Combining \eqref{addaddPS} and \eqref{jj'20}, we deduce that for any $j\in \NN$
$$\|u_j\|_{X_0}^2 \leq \kappa_*\left(1+\|u_j\|_{X_0}\right)$$
for a suitable positive constant $\kappa_*$\,. This proves that the Palais--Smale sequence $\{u_j\}_{j\in\NN}$ is bounded in $X_0$\,.

Hence Step~\ref{step1ps} is proved\,.}
\end{step}

\begin{step}\label{step2ps} Up to a subsequence, $\{u_j\}_{j\in\NN}$ strongly converges in $X_0$\,.
{\rm Since $\{u_j\}_{j\in\NN}$ is bounded in $X_0$ by Step~\ref{step1ps} and $X_0$ is a reflexive space
(being a Hilbert space, by \cite[Lemma~7]{svmountain}), up to a
subsequence, still denoted by $\{u_j\}_{j\in\NN}$, there exists $u_\infty \in X_0$
such that
\begin{equation}\label{convergenze0}
\begin{aligned}
 & \int_{\RR^n\times \RR^n}\big(u_j(x)-u_j(y)\big)\big(\varphi(x)-\varphi(y)\big) K(x-y)\,dx\,dy \to \\
& \qquad \qquad
\int_{\RR^n\times \RR^n}\big(u_\infty(x)-u_\infty(y)\big)\big(\varphi(x)-\varphi(y)\big)
K(x-y)\,dx\,dy  \quad \mbox{for any}\,\, \varphi\in X_0
\end{aligned}
\end{equation}
as $j\to +\infty$\,. Moreover, by \cite[Lemma~9]{servadeivaldinociBN}, up to a subsequence,
\begin{equation}\label{convergenze0bis}
\begin{aligned}
& u_j \to u_\infty \quad \mbox{in}\,\, L^2(\RR^n) \\
& u_j \to u_\infty \quad \mbox{in}\,\, L^q(\RR^n) \\
& u_j \to u_\infty \quad \mbox{a.e. in}\,\, \RR^n
\end{aligned}
\end{equation}
as $j\to +\infty$ and there exists $\ell\in L^q(\RR^n)$ such that
\begin{equation}\label{dominata20}
|u_j(x)|\leq \ell(x) \quad \mbox{a.e. in}\,\, \RR^n\,\quad \mbox{for any}\,\,j\in \NN
\end{equation}
(see, for instance, \cite[Theorem~IV.9]{brezis}).

By \eqref{crescita}, \eqref{convergenze0}--\eqref{dominata20}, the
fact that the map $t\mapsto f(\cdot, t)$ is continuous in $t\in \RR$
and the Dominated Convergence Theorem we get
\begin{equation}\label{convf0}
\int_\Omega f(x, u_j(x))u_j(x)\,dx \to \int_\Omega f(x, u_\infty(x))u_\infty(x)\,dx
\end{equation}
and
\begin{equation}\label{convfu0}
\int_\Omega f(x, u_j(x))u_\infty(x)\,dx \to \int_\Omega f(x, u_\infty(x))u_\infty(x)\,dx
\end{equation}
as $j\to +\infty$. Moreover, by \eqref{J'00} and Step~\ref{step1ps} we have that
$$\begin{aligned}
0\leftarrow & \langle \mathcal J_{K,\,\lambda}'(u_j), u_j\rangle\\
&= \int_{\RR^n\times \RR^n}|u_j(x)-u_j(y)|^2 K(x-y)\,dx\,dy -\lambda\int_\Omega |u_j(x)|^2\,dx - \int_\Omega f(x, u_j(x))u_j(x)\,dx
\end{aligned}$$
so that, by \eqref{convergenze0bis} and \eqref{convf0}, we can deduce that
\begin{equation}\label{convnorma10}
\int_{\RR^n\times \RR^n}|u_j(x)-u_j(y)|^2 K(x-y)\,dx\,dy\to
\lambda\int_\Omega |u_\infty(x)|^2\,dx + \int_\Omega f(x, u_\infty(x))u_\infty(x)\,dx
\end{equation}
as $j\to +\infty$.
Furthermore, again by \eqref{J'00}, we get
\begin{equation}\label{convj'0}
\begin{aligned}
0\leftarrow \langle \mathcal J_{K,\,\lambda}'(u_j), u_\infty\rangle & = \int_{\RR^n\times \RR^n}\big(u_j(x)-u_j(y)\big)\big(u_\infty(x)-u_\infty(y)\big) K(x-y)\,dx\,dy\\
 & \qquad -\lambda\int_\Omega u_j(x) u_\infty(x)\,dx - \int_\Omega f(x, u_j(x))u_\infty(x)\,dx
\end{aligned}
\end{equation}
as $j\to +\infty$.
By \eqref{convergenze0} with $\varphi=u_\infty$, \eqref{convergenze0bis}, \eqref{convfu0} and \eqref{convj'0} we obtain
\begin{equation}\label{convnorma20}
\begin{aligned}
\int_{\RR^n\times \RR^n}|u_\infty(x)-u_\infty(y)|^2 K(x-y)\,dx\,dy & = \lambda\int_\Omega |u_\infty(x)|^2\,dx\\
& \qquad +\int_\Omega f(x, u_\infty(x))u_\infty(x)\,dx\,.
\end{aligned}
\end{equation}
Thus \eqref{convnorma10} and \eqref{convnorma20} give us that
\begin{equation}\label{convnormaX00}
\|u_j\|_{X_0}\to \|u_\infty\|_{X_0}\,,
\end{equation}
as $j\to \infty$.

Finally, it is easy to see that
$$\begin{aligned}
\|u_j-u_\infty\|_{X_0}^2 & = \|u_j\|_{X_0}^2 + \|u_\infty\|_{X_0}^2 -2 \int_{\RR^n\times \RR^n} \big(u_j(x)-u_j(y)\big)\big(u_\infty(x)-u_\infty(y)\big) K(x-y)\,dx\,dy\\
& \to 2\|u_\infty\|_{X_0}^2-2\int_{\RR^n\times \RR^n}|u_\infty(x)-u_\infty(y)|^2 K(x-y)\,dx\,dy=0
\end{aligned}$$
as $j\to +\infty$, thanks to \eqref{convergenze0} and \eqref{convnormaX00}.
Therefore, the assertion of Step~\ref{step2ps} is proved. This concludes the proof of Proposition~\ref{propositionps1}.}
\end{step}
\end{proof}

Now, we are ready for proving Theorem~\ref{thmain1lambda}.

\subsection{Proof of Theorem~\ref{thmain1lambda}}\label{subsec:proofthmain1}
The idea consists in applying the Fountain Theorem. By Proposition~\ref{propositionps1}
we have that $\mathcal J_{K,\,\lambda}$ satisfies the Palais--Smale condition, while by \eqref{fsimmetrica}, we get that $\mathcal J_{K,\,\lambda}(-u)=\mathcal J_{K,\,\lambda}(u)$ for any $u \in X_0$. Then, it remains to study the geometry of the functional~$\mathcal J_{K,\,\lambda}$. For this purpose, we proceed by the following steps.
\begin{stepbis}\label{step1geometry} For any $k\in \NN$ there exists $r_k>0$ such that
$$a_{k}:=\max\Big\{\mathcal J_{K,\,\lambda}(u): u\in Y_k, \|u\|_{X_0}=r_{k}\Big\}\leq 0 \,.$$
{\rm By \eqref{Fmu}, we get that for any $u\in Y_k$
\begin{equation}\label{JYkreplaced}
\begin{aligned}
\mathcal J_{K,\,\lambda}(u) & \leq \frac{1}{2}\|u\|_{X_0}^{2}-\frac{\lambda}{2}\|u\|_{L^2(\Omega)}^{2}-a_3\|u\|_{L^{\mu}(\Omega)}^{\mu}+a_4|\Omega|\\
& \leq \frac{C_{k,\,\lambda}}{2}\|u\|_{X_0}^2-\hat C_{k,\,\mu}\|u\|_{X_0}^{\mu}+a_4|\Omega|
\end{aligned}
\end{equation}
for suitable positive constants  $C_{k,\,\lambda}$, depending on $k$ and $\lambda$, and $\hat C_{k,\,\mu}$, depending on $k$ and $\mu$\,. Here we used the fact that all the norms are equivalent in $Y_k$ .

As a consequence of \eqref{JYkreplaced}, we get that for any $u\in Y_k$ with $\|u\|_{X_0}=r_k$
$$\mathcal J_{K,\,\lambda}(u)\leq 0\,,$$
provided $r_k>0$ is large enough, due to the fact that $\mu>2$.
Thus Step~\ref{step1geometry} is proved.}
\end{stepbis}

\begin{stepbis}\label{lemma1}
Let $1\leq q<2^*$ and, for any $k\in \NN$, let
$$\beta _k:=\sup\Big\{\|u\|_{L^{q}(\Omega)}: u\in Z_k, \|u\|_{X_0}=1\Big\}\,.$$
Then $\beta_k\to 0$ as $k\to\infty$.

{\rm By definition of $Z_k$, we have that $Z_{k+1}\subset Z_k$ and so, as a consequence,
$0< \beta _{k+1}\leq \beta _k$ for any $k\in \NN$. Hence
\begin{equation}\label{betak}
\beta _k\to \beta
\end{equation}
as $k\to +\infty$, for some $\beta\geq 0$. Moreover, by definition of $\beta_k$, for any $k\in \NN$ there exists $u_k\in Z_k$ such that
\begin{equation}\label{uk}
\|u_{k}\|_{X_0}=1 \quad \mbox{and} \quad \|u_{k}\|_{L^{q}(\Omega)}>\beta_k/2\,.
\end{equation}

Since $X_0$ is a Hilbert space, and hence a reflexive Banach space, there exist $u_\infty\in X_0$ and a subsequence of $u_k$ (still denoted by $u_k$)
such that $u_{k}\to u_\infty$ weakly converges in $X_0$, that is
$$\langle u_k, \varphi \rangle_{X_0} \to \langle u_\infty, \varphi\rangle_{X_0} \qquad \mbox{for any}\,\,\, \varphi \in X_0$$
as $k\to +\infty$. Since ${\displaystyle \varphi=\sum_{j=1}^{+\infty}c_je_j}$\,, it follows that
$${\displaystyle \langle u_\infty, \varphi\rangle_{X_0}=\lim_{k\to +\infty}\langle u_k, \varphi \rangle_{X_0}= \lim_{k\to +\infty}\sum_{j=1}^{+\infty}c_j\langle u_k, e_j \rangle_{X_0}=0\,,}$$
thanks to the fact that the sequence $\{e_k\}_{k\in \NN}$ of eigenfunctions of $-\mathcal L_K$ is an orthogonal basis of $X_0$. Therefore
we can deduce that $u_\infty\equiv 0$. Hence by the Sobolev embedding theorem (see \cite[Lemma~9]{servadeivaldinociBN}), we get
\begin{equation}\label{ukLq}
u_{k}\to 0\quad \mbox{in}\,\,\, L^{q}(\Omega)
\end{equation}
as $k\to +\infty$. By \eqref{betak}, the fact that $\beta$ is nonnegative, and by \eqref{uk} and \eqref{ukLq},
we get that $\beta_k\to 0$ as $k\to +\infty$ and this concludes the proof of Step~\ref{lemma1}.}
\end{stepbis}

\begin{stepbis}\label{step2geometry} There exists $\gamma_k>0$ such that $$b_{k}:=\inf\Big\{\mathcal J_{K,\,\lambda}(u): u\in Z_k, \|u\|_{X_0}=\gamma_k\Big\}\to +\infty $$
as $k\to +\infty$\,.\\
{\rm By invoking \eqref{crescita} and integrating, it is easy to see that \eqref{cond22F} holds, and so, as a consequence, we get that there exists a constant $C>0$ such that
\begin{equation}\label{FC}
|F(x,t)|\leq C(1+|t|^{q})
\end{equation}
for any $x\in \overline \Omega$ and $t\in \RR$\,.
Then we obtain by \eqref{FC} for any $u\in Z_k\setminus\{0\}$
\begin{equation}\label{Jbetak}
\begin{aligned}
\mathcal J_{K,\,\lambda}(u)  & \geq \frac{1}{2}\|u\|_{X_0}^{2}-\frac \lambda 2\|u\|_{L^2(\Omega}^2-C\|u\|_{L^{q}(\Omega)}^{q}-C|\Omega| \\
& \geq C_{k,\,\lambda} \|u\|_{X_0}^2-C\Big\|\frac{u}{\|u\|_{X_0}}\Big\|_{L^q(\Omega)}^q \|u\|_{X_0}^q-C|\Omega|\\
& \geq C_{k,\,\lambda}\|u\|_{X_0}^{2}- C\beta_k^q\|u\|_{X_0}^q-C|\Omega|\\
& = \|u\|_{X_0}^{2}\left(C_{k,\,\lambda}-C\beta_k^q\|u\|_{X_0}^{q-2}\right)-C|\Omega|\,,
\end{aligned}
\end{equation}
where $\beta_k$ is defined as in Step~\ref{lemma1} and
$$C_{k,\,\lambda}=\begin{cases}
\frac 1 2 & \mbox{if} \quad \lambda\leq 0\\
\frac 1 2\left(1-\frac{\lambda}{\lambda_1}\right) & \mbox{if} \quad 0<\lambda<\lambda_1\\
\frac 1 2\left(1-\frac{\lambda}{\lambda_k}\right) & \mbox{if} \quad \lambda_k\leq \lambda<\lambda_{k+1}\,.\\
\end{cases}$$

Defining $\gamma_k$ as follows
$$\gamma_k=\Big(\frac{2C_{k,\,\lambda}}{q C\beta_k^q}\Big)^{1/(q-2)}\,,$$ it is easy to see that $\gamma_k\to +\infty$ as $k\to +\infty$, thanks to Step~\ref{lemma1}, the fact that $q>2$ and since $\{\lambda_k\}_{k\in \NN}$ is a divergent sequence. As a consequence of this and by \eqref{Jbetak} we get that for any $u\in Z_k$ with $\|u\|_{X_0}=\gamma_k$
$$\mathcal J_{K,\,\lambda}(u)\geq \|u\|_{X_0}^{2}\left(C_{k,\,\lambda}- C\beta_k^q\|u\|_{X_0}^{q-2}\right)-C|\Omega|=\left(1 -\frac 2 q\right)C_{k,\,\lambda}\gamma_k^2-C|\Omega|\to +\infty$$
as $k\to +\infty$. Thus Step~\ref{step2geometry} is completed.}
\end{stepbis}

Hence all the geometric features of the Fountain Theorem are satisfied and the proof of Theorem~\ref{thmain1lambda} is complete.

\medskip

We would like to emphasize that in the verification of the geometric structure of the functional~$\mathcal J_{K,\,\lambda}$ the Ambrosetti-Rabinowitz condition (namely, \eqref{Fmu}) was used only for proving Step~\ref{step1geometry}.

\section{Nonlinearities satisfying other superlinear conditions}\label{sec:noAR}
In this section we shall deal with problem~\eqref{problemaKlambda} when superlinear conditions on the term $f$ different from the Ambrosetti-Rabinowitz are satisfied.
In this framework we shall show that the functional~$\mathcal J_{K,\,\lambda}$ satisfies the Cerami condition, as well as the geometric requirements of the Fountain Theorem.

\subsection{Nonlinearities under the superlinear conditions~\eqref{Flimite} and \eqref{mathcalF}}\label{subsec:superlinear2}
First, we study the compactness properties of the functional~$\mathcal J_{K,\,\lambda}$, as stated in the following result:
\begin{proposition}\label{propositionps2}
Let $\lambda\in \RR$ and let $f:\overline \Omega \times \RR \to \RR$ be a function satisfying conditions
\eqref{fcontinua}, \eqref{crescita}, \eqref{fsimmetrica}, \eqref{Flimite} and \eqref{mathcalF}. Then, $\mathcal J_{K,\,\lambda}$ satisfies the Cerami condition at any level $c\in \RR$\,.
\end{proposition}
\begin{proof}
Let $c\in \RR$ and let $\{u_j\}_{j\in\NN}$ be a Cerami sequence in $X_0$, that is let $\{u_j\}_{j\in\NN}$ be such that
\begin{equation}\label{Jc0Cerami}
\mathcal J_{K,\,\lambda}(u_j)\to c
\end{equation}
and
\begin{equation}\label{J'00Cerami}
(1+\|u_j\|)\sup\Big\{ \big|\langle\,\mathcal J_{K,\,\lambda}'(u_j),\varphi\,\rangle \big|\,: \;
\varphi\in
X_0\,, \|\varphi\|_{X_0}=1\Big\}\to 0
\end{equation}
as $j\to +\infty$.

First, we show that the sequence~$\{u_j\}_{j\in\NN}$ is bounded in $X_0$. For this purpose we argue as in the proof of \cite[Lemma~2.2]{fangliu}. Suppose to the contrary that $\{u_j\}_{j\in\NN}$ is unbounded in $X_0$, that is suppose that, up to a subsequence, still denoted by $\{u_j\}_{j\in\NN}$,
\begin{equation}\label{contradictionCerami}
\|u_j\|_{X_0} \to +\infty
\end{equation}
as $j\to +\infty$\,.

By \eqref{J'00Cerami} and \eqref{contradictionCerami}, it is easy to see that
\begin{equation}\label{J'00CeramiPS}
\sup\Big\{ \big|\langle\,\mathcal J_{K,\,\lambda}'(u_j),\varphi\,\rangle \big|\,: \;
\varphi\in
X_0\,, \|\varphi\|_{X_0}=1\Big\}\to 0
\end{equation}
and so, also
\begin{equation}\label{J'00CeramiPSadd}
\sup\Big\{ \big|\langle\,\mathcal J_{K,\,\lambda}'(u_j),\varphi\,\rangle \big|\,: \;
\varphi\in
X_0\,, \|\varphi\|_{X_0}=1\Big\}\cdot \|u_j\|_{X_0}\to 0
\end{equation}
as $j\to +\infty$

Now, for any $j\in \NN$, let
\begin{equation}\label{vj}
v_j =\frac{u_j}{\|u_j\|_{X_0}}\,.
\end{equation}
Of course, the sequence $\{v_j\}_{j\in\NN}$ is bounded in $X_0$ and so, by \cite[Lemma~9]{servadeivaldinociBN}, up to a subsequence, there exists $v_\infty\in X_0$ such that
\begin{equation}\label{convergenzevj}
\begin{aligned}
& v_j \to v_\infty \quad \mbox{in}\,\, L^2(\RR^n) \\
& v_j \to v_\infty \quad \mbox{in}\,\, L^q(\RR^n) \\
& v_j \to v_\infty \quad \mbox{a.e. in}\,\, \RR^n
\end{aligned}
\end{equation}
as $j\to +\infty$ and there exists $\ell\in L^q(\RR^n)$ such that
\begin{equation}\label{dominatavj}
|v_j(x)|\leq \ell(x) \quad \mbox{a.e. in}\,\, \RR^n\,\quad \mbox{for any}\,\,j\in \NN
\end{equation}
(see \cite[Theorem~IV.9]{brezis}).
In the sequel we shall separately consider the cases when $v_\infty\equiv 0$ and $v_\infty\not \equiv 0$ and we shall prove that in both cases a contradiction occurs.

{\bf Case 1:} Suppose that
\begin{equation}\label{vinfty0}
v_\infty\equiv 0\,.
\end{equation}
As in \cite{jeanjean}, we can say that for any $j\in \NN$ there exists $t_j \in[0,1]$ such that
\begin{equation}\label{maxtj}
\mathcal J_{K,\,\lambda}(t_ju_j) = \max_{t\in [0,1]}\mathcal J_{K,\,\lambda}(tu_j)\,.
\end{equation}
Since \eqref{contradictionCerami} holds, for any $m\in \NN$, we can choose $r_m=2\sqrt{m}$ such that \begin{equation}\label{rm}
r_m\|u_j\|_{X_0}^{-1} \in (0,1)\,,
\end{equation}
provided $j$ is large enough, say $j>\bar \jmath$\,, with $\bar \jmath=\bar \jmath\,(m)$.

By \eqref{convergenzevj}, \eqref{vinfty0} and the continuity of the function $F$, we get that
\begin{equation}\label{addlambda}
\int_\Omega |r_mv_j(x)|^2\,dx \to 0
\end{equation}
and
\begin{equation}\label{Fae}
F(x, r_mv_j(x))\to F(x, r_mv_\infty(x))\quad \mbox{a.e.}\,\,\, x\in \Omega
\end{equation}
as $j\to +\infty$ for any $m\in \NN$\,. Moreover, integrating \eqref{crescita} and taking into account \eqref{dominatavj}, we have that
\begin{equation}\label{dominataF}
\begin{aligned}
|F(x,r_mv_j(x))| & \leq a_1\,|r_mv_j(x)|+\frac{a_2}{q}\, |r_mv_j(x)|^{q}\\
& \leq a_1\,r_m\ell(x)+\frac{a_2}{q}\, (r_m\ell(x))^{q}\in L^1(\Omega)\,,
\end{aligned}
\end{equation}
a.e. $x\in \Omega$ and for any $m, j\in \NN$\,.
Hence \eqref{Fae}, \eqref{dominataF} and the Dominated Convergence Theorem yield that
\begin{equation}\label{FL1}
F(\cdot, r_mv_j(\cdot))\to F(\cdot, r_mv_\infty(\cdot))\quad \mbox{in}\,\,\, L^1(\Omega)
\end{equation}
as $j\to +\infty$ for any $m\in \NN$\,. Since $F(x, 0)=0$ for any $x\in \overline \Omega$ and \eqref{vinfty0} holds, \eqref{FL1} gives that
\begin{equation}\label{add33}
\int_{\Omega}F(x, r_mv_j(x))\,dx \to 0
\end{equation}
as $j\to +\infty$ for any $m\in \NN$\,.
Thus \eqref{maxtj}, \eqref{rm}, \eqref{addlambda} and \eqref{add33} yield
$$\begin{aligned}
\mathcal J_{K,\,\lambda}(t_ju_j) &\geq \mathcal J_{K,\,\lambda}(r_m\|u_j\|_{X_0}^{-1}u_j)\\
& = \mathcal J_{K,\,\lambda}(r_mv_j)\\
& = \frac 1 2 \|r_mv_j\|_{X_0}^2-\frac \lambda 2\int_\Omega |r_mv_j(x)|^2\,dx- \int_\Omega F(x, r_mv_j(x))\,dx\\
& = 2m -\frac \lambda 2\int_\Omega |r_mv_j(x)|^2\,dx - \int_{\Omega}F(x,r_mv_j(x))\,dx \geq m\,,
\end{aligned}$$
provided $j$ is large enough and for any $m\in \NN$. From this we deduce that
\begin{equation}\label{limiteJtj}
\mathcal J_{K,\,\lambda}(t_ju_j) \to +\infty
\end{equation}
as $j\to +\infty$.

Now, note that $\mathcal J_{K,\,\lambda}(0) = 0$ and \eqref{Jc0Cerami} holds. Combining these two facts and \eqref{limiteJtj}, it is easily seen that  $t_j \in (0,1)$ and so by \eqref{maxtj}, we get that
$$\frac{d}{dt}\Big|_{t=t_j}\mathcal J_{K,\,\lambda}(tu_j) = 0$$
for any $j\in \NN$\,. As a consequence of this, we have that
\begin{equation}\label{tjderivata}
\langle \mathcal J_{K,\,\lambda}'(t_ju_j), t_ju_j\rangle = t_j\frac{d}{dt}\Big|_{t=t_j}\mathcal J_{K,\,\lambda}(tu_j) = 0\,.
\end{equation}

We claim that
\begin{equation}\label{add34}
\limsup_{j\to +\infty}\mathcal J_{K,\,\lambda}(t_ju_j)\leq \kappa\,,
\end{equation}
for a suitable positive constant $\kappa$\,.
Before proving this fact, we note that, as a consequence of the assumptions~\eqref{fsimmetrica} and \eqref{mathcalF}, the following condition is satisfied:
\begin{equation}\label{mathcalFF}
\begin{aligned}
& \quad \quad \mbox{there exists}\,\, \gamma\geq 1\,\, \mbox{such that for any}\,\,\, x\in \Omega\\
& \mathcal F(x,t') \leq \gamma \mathcal F(x,t)\,\, \mbox{for any}\,\, t, t'\in \RR\,\, \mbox{with}\,\, 0<|t'|\leq |t|\,,
\end{aligned}
\end{equation}
where $\mathcal F$ is the function given by \eqref{mathcalFdef}.

Now, by invoking \eqref{tjderivata} and using \eqref{mathcalFF}, we get
$$\begin{aligned}
\frac{1}{\gamma}\mathcal J_{K,\,\lambda}(t_ju_j) &= \frac{1}{\gamma} \left(\mathcal J_{K,\,\lambda}(t_ju_j) - \frac{1}{2}\langle \mathcal J_{K,\,\lambda}'(t_ju_j), t_ju_j\rangle \right)\\
& = \frac 1 \gamma \left(-\int_\Omega F(x, t_ju_j(x))\,dx+\frac 1 2 \int_\Omega t_ju_j(x)\,f(x,t_ju_j(x))\,dx\right)\\
&=  \frac{1}{\gamma} \int_{\Omega}\mathcal {F}(x,t_ju_j(x))\,dx\\
&\leq  \int_{\Omega}\mathcal {F}(x,u_j(x))\,dx \\
& = \int_\Omega \left(\frac 1 2\, u_j(x)f(x,u_j(x))-F(x, u_j(x))\right)\,dx\\
&= \mathcal J_{K,\,\lambda}(u_j) - \frac 1 2 \langle \mathcal J_{K,\,\lambda}'(u_j), u_j\rangle
\to c
\end{aligned}$$
as $j\to +\infty$\,, thanks to \eqref{Jc0Cerami} and \eqref{J'00CeramiPSadd}. This proves \eqref{add34}, which contradicts \eqref{limiteJtj}\,. Thus, the sequence~$\{u_j\}_{j\in\NN}$ has to be bounded in $X_0$.

{\bf Case 2:} Suppose that
\begin{equation}\label{vinftynot0}
v_\infty\not \equiv0\,.
\end{equation}
Then the set $\Omega':= \{x\in \Omega : v_\infty(x) \neq 0\}$ has positive Lebesgue measure and
\begin{equation}\label{ujinfty}
|u_j(x)| \to +\infty \,\,\,\, \mbox{a.e.}\,\, x\in \Omega'
\end{equation}
as $j\to +\infty$, thanks to \eqref{contradictionCerami}, \eqref{vj}, \eqref{convergenzevj} and \eqref{vinftynot0}.

By \eqref{Jc0Cerami} and \eqref{contradictionCerami} it is easy to see that
$$\frac{\mathcal J_{K,\,\lambda}(u_j)}{\|u_j\|_{X_0}^2}\to 0\,,$$
that is
\begin{equation}\label{convJ0replaced}
\frac 1 2 -\frac \lambda 2 \int_\Omega \frac{|u_j(x)|^2}{\|u_j\|_{X_0}^2}\,dx-\int_{\Omega'} \frac{F(x,u_j(x))}{\|u_j\|_{X_0}^2}\,dx-\int_{\Omega\setminus \Omega'} \frac{F(x,u_j(x))}{\|u_j\|_{X_0}^2}\,dx=o(1)
\end{equation}
as $j\to +\infty$\,.

Now, observe that, by the variational characterization of the first eigenvalue $\lambda_1$ of $-\mathcal L_K$ (see \cite[Proposition~9]{svlinking}), that is
$$\lambda_1=\min_{u\in X_0\setminus\{0\}}\frac{{\displaystyle\int_{\RR^n\times \RR^n }
|u(x)-u(y)|^2K(x-y) dx\,dy}}{{\displaystyle\int_\Omega |u(x)|^2\,dx}}\,,$$
we get that for any $u\in X_0$
\begin{equation}\label{mMlambda}
\|u\|_{L^2(\Omega)}^2\leq \frac{1}{\lambda_1} \|u\|_{X_0}^2\,.
\end{equation}
Hence, by \eqref{convJ0replaced} and \eqref{mMlambda}, we can deduce that
\begin{equation}\label{convJ0replaced2}
\begin{aligned}
o(1) & = \frac 1 2 -\frac \lambda 2 \int_\Omega \frac{|u_j(x)|^2}{\|u_j\|_{X_0}^2}\,dx -\int_{\Omega'} \frac{F(x,u_j(x))}{\|u_j\|_{X_0}^2}\,dx-\int_{\Omega\setminus \Omega'} \frac{F(x,u_j(x))}{\|u_j\|_{X_0}^2}\,dx \\
& \leq \frac 1 2\max\Big\{1, 1-\frac{\lambda}{\lambda_1}\Big\} -\int_{\Omega'} \frac{F(x,u_j(x))}{\|u_j\|_{X_0}^2}\,dx-\int_{\Omega\setminus \Omega'} \frac{F(x,u_j(x))}{\|u_j\|_{X_0}^2}\,dx
\end{aligned}
\end{equation}
as $j\to +\infty$\,.

Let us consider separately the two integrals from formula~\eqref{convJ0replaced2}. With respect to the first one, we have that
$$\begin{aligned}
\frac{F(x,u_j(x))}{\|u_j\|_{X_0}^2} & = \frac{F(x,u_j(x))}{|u_j(x)|^2}\frac{|u_j(x)|^2}{\|u_j\|_{X_0}^2} \\
& =\frac{F(x,u_j(x))}{|u_j(x)|^2}|v_j(x)|^2 \to +\infty\,\,\,\,\,\, \mbox{a.e.}\,\, x\in \Omega'
\end{aligned}$$
as $j\to +\infty$, thanks to \eqref{Flimite}, \eqref{convergenzevj}, \eqref{ujinfty} and the definition of $\Omega'$\,. Hence, by using the Fatou lemma, we obtain
\begin{equation}\label{vnot0}
\int_{\Omega'} \frac{F(x,u_j(x))}{\|u_j\|_{X_0}^2}\,dx \to +\infty
\end{equation}
as $j\to +\infty$\,.

As for the second integral from \eqref{convJ0replaced2}, we claim that
\begin{equation}\label{v=0}
\lim_{j\to +\infty}\int_{\Omega\setminus \Omega'} \frac{F(x,u_j(x))}{\|u_j\|_{X_0}^2}\,dx\geq 0
\end{equation}
(note that this limit exists thanks to \eqref{convJ0replaced2} and \eqref{vnot0}).
Indeed by \eqref{Flimite}, it follows that
\begin{equation}\label{G2}
\lim_{|t|\rightarrow +\infty}F(x,t)=+\infty\,\,\, \mbox{uniformly for any}\,\, x\in \overline\Omega\,.
\end{equation}
Hence, by \eqref{G2} there exist two positive constants $\tilde t$ and $H$ such that
\begin{equation}\label{G3}
F(x,t)\geq H
\end{equation}
for every $x\in \overline{\Omega}$ and $|t|> \tilde t$. On the other hand, since $F$ is continuous in $\overline \Omega\times \RR$, one has
\begin{equation}\label{G4}
F(x,t)\geq \min_{(x,t)\in \overline{\Omega}\times[-\tilde t, \tilde t]}F(x,t)\,,
\end{equation}
for every $x\in \overline{\Omega}$ and $|t|\leq \tilde t$\,.
Then it follows that by \eqref{G3} and \eqref{G4}
\begin{equation}\label{G5}
F(x,t)\geq \kappa\,\,\,\,\, \mbox{for any}\,\,\, (x,t)\in \overline{\Omega}\times \RR
\end{equation}
for some constant $\kappa$.  By \eqref{contradictionCerami} and \eqref{G5} the claim now follows.

In conclusion, by \eqref{convJ0replaced2}, \eqref{vnot0} and \eqref{v=0} we get a contradiction.
Thus the sequence~$\{u_j\}_{j\in\NN}$ is bounded in $X_0$.

In order to complete the proof of Proposition~\ref{propositionps2} we can argue from now on as in Step~\ref{step2ps} of the proof of Proposition~\ref{propositionps1}.
\end{proof}

We remark that in the proof of Proposition~\ref{propositionps2} the assumption~\eqref{mathcalF} was used (and was crucial) only for proving the inequality~\eqref{add34}.

\subsection{Proof of Theorem~\ref{thmain2lambda}}\label{subsec:proofthmain2} By Proposition~\ref{propositionps2} and \eqref{fsimmetrica}, we have that $\mathcal J_{K,\,\lambda}$ satisfies the Cerami condition (and hence also the Palais--Smale condition) and $\mathcal J_{K,\,\lambda}(-u)=\mathcal J_{K,\,\lambda}(u)$ for any $u\in X_0$.
The verification of the geometric assumption~$(ii)$ of the Fountain Theorem follows as in Step~\ref{step2geometry} in Subsection~\ref{subsec:proofthmain1}. It remains to verify the condition~$(i)$. For this purpose we shall use the finite-dimensionality of the linear subspace~$Y_k$ and assumption~\eqref{Flimite}\,.

Indeed, by~\eqref{Flimite} for any $\varepsilon>0$, there exists $\delta_\varepsilon>0$ such that
\begin{equation}\label{add42}
F(x,t)\geq \varepsilon\,|t|^2\,\, \mbox{for any}\,\,x\in \overline \Omega\,\, \mbox{and any}\,\, t\in \RR\,\, \mbox{with}\,\, |t|> \delta_\varepsilon\,,
\end{equation}
while, by the Weierstrass Theorem, we have that
\begin{equation}\label{FWeirstrass}
{\displaystyle F(x,t)\geq m_\varepsilon:=\min_{x\in \overline \Omega, |t|\leq \delta_\varepsilon} F(x,t)}
\,\, \mbox{for any}\,\,x\in \overline \Omega\,\, \mbox{and any}\,\, t\in \RR\,\, \mbox{with}\,\, |t|\leq \delta_\varepsilon\,.
\end{equation}
Note that $m_\varepsilon\leq 0$, since $F(x,0)=0$ for any $x\in \overline \Omega$\,.
By \eqref{add42} and \eqref{FWeirstrass}, it is easy to see that
$$F(x,t)\geq \varepsilon|t|^2-B_\varepsilon\,\,\,\, \mbox{for any}\,\, (x,t)\in \overline\Omega\times \RR$$
for a suitable positive constant $B_\varepsilon$ (say, $B_\varepsilon\geq \varepsilon \delta_\varepsilon^2-m_\varepsilon$)\,.

As a consequence of this and by the fact that $Y_k$ is finite-dimensional, we have for any $u \in Y_k$
\begin{equation}\label{add44replaced}
\begin{aligned}
\mathcal J_{K,\,\lambda}(u) & =\frac {1} {2} \|u\|_{X_0}^2-\frac {\lambda} {2} \|u\|_{L^2(\Omega)}^2-\int_{\Omega} F(x, u(x))\,dx \\
& \leq C_{k,\,\lambda}\|u\|_{X_0}^{2}- \varepsilon\|u\|_{L^2(\Omega)}^2+B_\varepsilon|\Omega|\\
& \leq \left( C_{k,\,\lambda}- \varepsilon C_k\right)\|u\|_{X_0}^2+B_\varepsilon|\Omega|\,,
\end{aligned}
\end{equation}
where $C_{k,\,\lambda}$ and $C_k$ are positive constants, the first one depending on $k$ and $\lambda$ and the second one only on $k$.
Hence, choosing $\varepsilon$ such that $C_{k,\,\lambda}- \varepsilon C_k<0$, we get that for any $u\in Y_k$ with $\|u\|_{X_0}=r_k$
$$\mathcal J_{K,\,\lambda}(u)\leq 0\,,$$
provided $r_k>0$ is large enough. This proves that $\mathcal J_{K,\,\lambda}$ satisfies condition~$(i)$ of the Fountain Theorem and this completes the proof of Theorem~\ref{thmain2lambda}.

\subsection{Nonlinearities satisfying the superlinear conditions~\eqref{Flimite} and \eqref{fmonotonar}}\label{subsec:superlinear3}
In this setting we need the following lemma, whose proof was given in \cite[Lemma~2.3]{liu2}: it will be crucial in the proof of Theorem~\ref{thmain3lambda}.
\begin{lemma}\label{lemmaps3}
If \eqref{fmonotonar} holds, then for any $x\in \Omega$, the function $\mathcal F(x,t)$ is increasing when $t\geq \bar t$
and decreasing when $t\leq -\bar t$, where $\mathcal F$ is the function given by \eqref{mathcalFdef}.

In particular, there exists $C_{1}>0$ such that
$$\mathcal {F}(x,s)\leq \mathcal {F}(x,t)+C_{1}$$
for any $x\in\Omega$ and $0\leq s\leq t$ or $t\leq s\leq 0$\,.
\end{lemma}

\begin{proposition}\label{propositionps3}
Let $\lambda\in \RR$ and let $f:\overline \Omega \times \RR \to \RR$ be a function satisfying conditions
\eqref{fcontinua}, \eqref{crescita}, \eqref{Flimite} and \eqref{fmonotonar}. Then, $\mathcal J_{K,\,\lambda}$ satisfies the Cerami condition at any level $c\in \RR$\,.
\end{proposition}
\begin{proof}
We can argue exactly as in the proof of Proposition~\ref{propositionps2}. We only have to modify the proof of inequality~\eqref{add34}: indeed, for proving it, in Proposition~\ref{propositionps2} we used condition~\eqref{mathcalF} (actually \eqref{mathcalFF}), which is now no more assumed.

Here we will show the validity of \eqref{add34} by making use of the assumption~\eqref{fmonotonar} and of Lemma~\ref{lemmaps3}. We point out that our notation is the one used in the proof of Proposition~\ref{propositionps2}.
In view of Lemma~\ref{lemmaps3} we have that
$$\begin{aligned}
\mathcal J_{K,\,\lambda}(t_ju_j) &= \mathcal J_{K,\,\lambda}(t_ju_j) - \frac{1}{2}\langle \mathcal J_{K,\,\lambda}'(t_ju_j), t_ju_j\rangle \\
& =  \int_{\Omega}\mathcal {F}(x,t_ju_j(x))\,dx \\
& = \int_{\{u_j\geq 0\}}\mathcal {F}(x,t_ju_j(x))\,dx + \int_{\{u_j< 0\}}\mathcal {F}(x,t_ju_j(x))\,dx \\
& \leq  \int_{\{u_j\geq 0\}}\Big[\mathcal {F}(x,u_j(x))+ C_{1}\Big]
+ \int_{\{u_j<0\}}\Big[\mathcal {F}(x,u_j(x))+ C_{1}\Big]\\
& =  \int_{\Omega}\mathcal {F}(x,u_j(x))\,dx + C_{1}|\Omega| \\
& = \mathcal J_{K,\,\lambda}(u_j) - \frac{1}{2}\, \langle \mathcal J_{K,\,\lambda}'(u_j), u_j\rangle + C_{1}|\Omega|
\to c + C_1|\Omega|
\end{aligned}$$
as $j\to +\infty$\,. This proves \eqref{add34}.
The proof of Proposition~\ref{propositionps3} is thus completed.
\end{proof}

\subsection{Proof of Theorem~\ref{thmain3lambda}}\label{subsec:proofthmain3}
The functional~$\mathcal J_{K,\,\lambda}$ satisfies the Cerami condition by Proposition~\ref{propositionps3}, and so also the Palais-Smale assumption is satisfied. Moreover, $\mathcal J_{K,\,\lambda}(-u)=\mathcal J_{K,\,\lambda}(u)$ for any $u\in X_0$, thanks to \eqref{fsimmetrica}.

As for the geometric features of $\mathcal J_{K,\,\lambda}$, condition~$(ii)$ of the Fountain Theorem follows as in Step~\ref{step2geometry} of the proof of Theorem~\ref{thmain1lambda}, whereas condition~$(i)$ can be proved as in the proof of Theorem~\ref{thmain2lambda}.
Hence, the assertion of Theorem~\ref{thmain3lambda} is obtained.


\begin{thebibliography}{99}
\bibitem{adams} {\sc R.A. Adams}, {\em Sobolev Spaces}, Academic Press, New York, 1975.

\bibitem{alvesliu} {\sc C.O. Alves and S.B. Liu}, {\em On superlinear $p(x)$-Laplacian equations in
$\RR^N$}, Nonlinear Anal., 73, no.~8, 2566--2579 (2010).

\bibitem{ar} {\sc A. Ambrosetti and P. Rabinowitz},
{\em Dual variational methods in critical point theory and
applications}, {J. Funct. Anal.}, 14, 349--381 (1973).


\bibitem{bartsch} {\sc T. Bartsch}, {\em Infinitely many solutions of a symmetric Dirichlet problem}, Nonlinear Anal., 20, no.~1, 1205--1216 (1993).

\bibitem{molicaservadeizhang}  {\sc Z. Binlin, G. Molica Bisci and R. Servadei}, {\em Superlinear nonlocal fractional problems with infinitely many solutions}, {Nonlinearity}, 28, no.~7, 2247-2264 (2015).

\bibitem{brezis} {\sc H. Brezis}, Analyse fonctionelle. Th\'{e}orie et
applications, {\em Masson}, Paris, 1983.

\bibitem{cerami1} {\sc G. Cerami},
{\em An existence criterion for the critical points on unbounded manifolds}, (Italian) Istit. Lombardo Accad. Sci. Lett. Rend. A 112, no.~2, 332--336 (1978).

\bibitem{cerami2} {\sc G. Cerami},
{\em On the existence of eigenvalues for a nonlinear boundary value problem}, (Italian) Ann. Mat. Pura Appl. (4) 124, 161--179 (1980).

\bibitem{CostaM} {\sc D.G. Costa, C.A. Magalh\~{a}es},
{\em
Variational elliptic problems which are nonquadratic at infinity}, Nonlinear Anal. 23,
1401--1412 (1994).

\bibitem{valpal} {\sc E. Di Nezza, G. Palatucci and E. Valdinoci},
{\em Hitchhiker's guide to the fractional Sobolev spaces},
Bull. Sci. Math., 136, no.~5, 521--573 (2012).

\bibitem{fangliu} {\sc F. Fang and S.B. Liu}, {\em Nontrivial solutions of superlinear p-Laplacian equations}, J. Math. Anal. Appl., 351, no.~1, 138--146 (2009).

\bibitem{fmz} {\sc M. Ferrara, G. Molica Bisci and B. Zhang}, {\em Existence of weak solutions for non-local fractional problems via Morse theory}, Discrete and Continuous Dynamical Dystems Series B, 19, 2493--2499 (2014).

\bibitem{jeanjean} {\sc L. Jeanjean}, {\em On the existence of bounded Palais-Smale sequences and application to a Landesman-Lazer type problem set on
$\RR^N$}, Proc. Roy. Soc. Edinburgh Sect. A, 129, no.~4, 787--809 (1999).

\bibitem{liu} {\sc S.B. Liu}, {\em On ground states of superlinear $p$-Laplacian equations in $\RR^N$}, J. Math. Anal. Appl., 361, no.~81, 48--58 (2010).

\bibitem{liu2} {\sc S.B. Liu}, {\em On superlinear problems without Ambrosetti-Rabinowitz condition}, Nonlinear Anal., 73, no.~3, 788--795 (2010).

\bibitem{liuliu} {\sc S.B. Liu and S.J. Liu}, {\em Infinitely many solutions for a superlinear elliptic equation}, Acta. Math. Sinica (Chin. Ser.), 46, no.~4, 625--630 (2003).

\bibitem{miyagaki} {\sc O. Miyagaki and M. Souto}, {\em Superlinear problems without Ambrosetti and Rabinowitz growth condition}, J. Differential Equations, 245, no.~12, 3628--3638 (2008).

\bibitem{rabinowitz} {\sc P.H. Rabinowitz}, Minimax Methods in Critical Point Theory with Applications to Differential Equations,  {\em CBMS Reg. Conf. Ser. Math.}, 65, {\em American Mathematical Society}, Providence, RI, 1986.

\bibitem{schechter} {\sc M. Schechter and W. Zou}, {\em Superlinear problems}, Pacific J. Math., 214, no.~1, 145--160 (2004).

\bibitem{sY}{\sc R. Servadei}, {\em The Yamabe equation in a non-local setting},
Adv. Nonlinear Anal. 2, 235--270 (2013).

\bibitem{servadeiKavian}{\sc R. Servadei}, {\em Infinitely many solutions for fractional Laplace equations with subcritical nonlinearity}, Contemp. Math., 595, 317--340 (2013).

\bibitem{sv}{\sc R. Servadei and E. Valdinoci}, {\em Lewy-Stampacchia type estimates for variational
inequalities driven by $($non$)$local operators}, Rev. Mat. Iberoam. 29, no.~3, 1091--1126 (2013).

\bibitem{svmountain}{\sc R. Servadei and E. Valdinoci}, {\em Mountain Pass solutions
for non-local elliptic operators}, J. Math. Anal. Appl., 389, 887--898 (2012).

\bibitem{svlinking}{\sc R. Servadei and E. Valdinoci}, {\em Variational methods for non-local operators of elliptic type}, Discrete Contin. Dyn. Syst. 33, no.~5, 2105--2137 (2013).

\bibitem{servadeivaldinociBN}{\sc R. Servadei and E. Valdinoci}, {\em The Brezis-Nirenberg result for the fractional Laplacian}, Trans. Amer. Math. Soc., 367, no.~1, 67--102 (2015).

\bibitem{servadeivaldinociBNLOW}{\sc R. Servadei and E. Valdinoci}, {\em A Brezis-Nirenberg result for non-local critical equations in low dimension}, Commun. Pure Appl. Anal. 12, no.~6, 2445--2464 (2013).

\bibitem{struwe} {\sc M. Struwe}, Variational Methods,
Applications to Nonlinear Partial Differential Equations and Hamiltonian Systems,
{\em Ergebnisse der Mathematik und ihrer Grenzgebiete},  3,
{\em Springer Verlag}, Berlin--Heidelberg, 1990.

\bibitem{willem} {\sc M. Willem}, Minimax Theorems, Birkh\"{a}user, Basel, 1996.

\bibitem{WZ} {\sc M. Willem and W. Zou}, {\em On a Schr\"{o}dinger equation with periodic potential and spectrum point zero}, Indiana Univ.
Math. J. 52, 109--132 (2003).

\bibitem{Z} {\sc W.M. Zou}, {\em Variant fountain theorems and their applications}, Manuscripta Math. 104, 343--358 (2001).
\end{thebibliography}
\end{document}